\documentclass[12pt]{amsart}
\usepackage{etex}

\DeclareMathAlphabet{\mathcalligra}{T1}{calligra}{m}{n}
\makeindex

\usepackage{graphicx,tikz,url}
\usepackage[latin1]{inputenc} \usepackage[T1]{fontenc} \usepackage{lmodern}
\usepackage{pgf}
\usepackage{epsfig,graphicx,color,amsmath,a4wide,amssymb,amsfonts,amsbsy,xy,latexsym,epsf,pstricks,verbatim,times}
\usepackage{MnSymbol}
\usepackage{pgf,color}

\usepackage{changebar}

\newcommand*{\DashedArrow}[1][]{\mathbin{\tikz [baseline=-0.25ex,-latex, dashed,#1] \draw [#1] (0pt,0.5ex) -- (1.3em,0.5ex);}}%
\definecolor{LightGrey}{rgb}{.85,.85,.85}
\definecolor{DarkGrey}{rgb}{.5,.5,.5}

\definecolor{Blue}{rgb}{.0,.0,0.9}
\definecolor{LightBlue1}{rgb}{.2,.4,0.9}
\definecolor{LightBlue2}{rgb}{.3,.5,0.9}
\definecolor{LightBlue3}{rgb}{.4,.6,0.9}
\definecolor{LightBlue4}{rgb}{.5,.7,.9}
\definecolor{LightBlue5}{rgb}{.6,.8,.9}
\definecolor{LightBlue6}{rgb}{.7,.9,.9}

\definecolor{Red}{rgb}{.9,.0,.0}
\definecolor{LightRed1}{rgb}{0.9,.2,.4}
\definecolor{LightRed2}{rgb}{0.9,.3,.5}
\definecolor{LightRed3}{rgb}{0.9,.4,.6}
\definecolor{LightRed4}{rgb}{.9,.5,.7}
\definecolor{LightRed5}{rgb}{.9,.6,.8}
\definecolor{LightRed6}{rgb}{.9,.7,.9}  

\def\ttimes{{.}}
\def\FAIL{\mathbf{FAIL}}

\def\thetachar{{\vartheta}}
\def\dotu{{\dot{u}}}
\def\dotx{{\dot{x}}}

\def\bw{{\boldsymbol{w}}}
\def\OO{{\mathfrak{O}}}

\def\KK{{\mathbf K}}
\def\barKK{{ {\bar \KK}}}

\definecolor{Grey}{rgb}{.5,.5,.5}

\definecolor{Blue}{rgb}{.0,.0,0.9}
\definecolor{LightBlue1}{rgb}{.2,.4,0.9}
\definecolor{LightBlue2}{rgb}{.3,.5,0.9}
\definecolor{LightBlue3}{rgb}{.4,.6,0.9}
\definecolor{LightBlue4}{rgb}{.5,.7,.9}
\definecolor{LightBlue5}{rgb}{.6,.8,.9}
\definecolor{LightBlue6}{rgb}{.7,.9,.9}

\definecolor{Red}{rgb}{.9,.0,.0}
\definecolor{LightRed1}{rgb}{0.9,.2,.4}
\definecolor{LightRed2}{rgb}{0.9,.3,.5}
\definecolor{LightRed3}{rgb}{0.9,.4,.6}
\definecolor{LightRed4}{rgb}{.9,.5,.7}
\definecolor{LightRed5}{rgb}{.9,.6,.8}
\definecolor{LightRed6}{rgb}{.9,.7,.9}

\xyoption{all}
\usepackage[english]{babel}

\newcounter{noalgo}[section]
\newdimen\indentalgo
\newdimen\indentalgodec\indentalgo=0.0mm\indentalgodec=10mm

\newcommand{\If}{\advance\indentalgo by \indentalgodec {\bf if }}

\newcommand{\For}{\global\advance\indentalgo by \indentalgodec {\bf for }}

\newcommand{\Endindent}{\global\advance\indentalgo by -\indentalgodec}

\newdimen\decalage \decalage=0.5cm
\newcounter{algo} 

\newcommand{\PP}{\mathbf P}

\def\<<{\leavevmode
  \raise0.28ex\hbox{$\scriptscriptstyle\langle\!\langle$}\nobreak
  \hskip -.6pt plus.3pt minus.2pt\,}
\def\>>{\,\nobreak\hskip -.6pt plus.3pt minus.2pt
  \raise0.28ex\hbox{$\scriptscriptstyle\rangle\!\rangle$}}

\def\Hom{\mathop{\rm{Hom}}\nolimits }

\def\Pic{\mathop{\rm{Pic}}\nolimits }
\def\Spec{\mathop{\rm{Spec}}\nolimits }

\def\bS{{\bf S  }}
\def\bT{{\bf T  }}
\def\bP{{\bf P  }}
\def\bQ{{\bf Q  }}
\def\bR{{\bf R  }}

\def\bs{{\bf s  }}
\def\bp{{\bf p  }}
\def\bq{{\bf q  }}
\def\br{{\bf r  }}

\def\bS{{\bf S  }}

\def\Gm{{{\bf  G }_m}}

\def\LL{{\mathbf L}}

\def\GG{{\mathbf G}}

\def\aa{{\mathbf a}}

\def\GG{{\mathbf G}}

\def\ZZ{{\mathbf Z}}

\def\ugot{{\mathfrak u}}
\def\vgot{{\mathfrak v}}
\def\egot{{\mathfrak e}}

\def\cG{{\mathcal G}}

\def\cL{{\mathcal L}}

\def\cM{{\mathcal M}}

\def\cO{{\mathcal O}}

\newtheorem{lemma}{Lemma}

\newtheorem{proposition}{Proposition}
\newtheorem{theorem}{Theorem}

\providecommand{\myproofname}{Proof}

\usepackage{enumitem}

\begin{document}

\begin{abstract}
We show how to efficiently evaluate functions on Jacobian varieties and
their quotients. We deduce an  algorithm to compute $(l,l)$ isogenies
between Jacobians of genus two curves in  quasi-linear time  
in the degree $l^{2}$.
\end{abstract}

\title{Computing functions on Jacobians and their quotients}

\author{Jean-Marc Couveignes}
\address{Jean-Marc Couveignes,   Univ. Bordeaux,
Bordeaux INP, CNRS, IMB, UMR 5251, F-33400 Talence, France.}
\address{Jean-Marc Couveignes, INRIA, LFANT, LIRIMA, F-33400 Talence, France.}
\email{Jean-Marc.Couveignes@u-bordeaux.fr}

\author{Tony Ezome}
\address{Tony Ezome, Universit{\'e} des Sciences et Techniques de Masuku,
Facult{\'e} des Sciences, D{\'e}partement de math{\'e}matiques et informatique,
BP 943 Franceville, Gabon.}%
\address{Tony Ezome, INRIA, LIRIMA}
\email{latonyo2000@yahoo.fr}

\date{\today}

\maketitle
\setcounter{tocdepth}{2} 
\tableofcontents

\section{Introduction}

We consider the problem of computing the quotient of the Jacobian
variety $J$ of a curve $C$ by a maximal isotropic subgroup $V$ 
in its $l$-torsion for $l$ an odd prime integer. The genus one case   has been explorated a lot  since Vélu \cite{velu1,velu2}. A recent bibliography
can be found in \cite{bostan}.
In this work we first study this problem in general, showing how
to quickly design and evaluate standard functions (including Theta functions)
 on the quotient $J/V$. We then turn to the specific case when the dimension $g$ of
$J$ equals two. In that case, the quotient is, at least generically,  the Jacobian of another curve $D$. The quotient isogeny 
induces a map from $C$ into the Jacobian of $D$
that can then be described in a compact form: a few  rational 
fractions of degree $O(l)$.
We explain how to compute $D$ and the
map from  $C$ into the Jacobian of $D$ in  quasi-linear time in the degree $\# V = l^2$.

{\bf Plan} In Section~\ref{sec:fj} we bound the complexity of evaluating standard functions on Jacobians, including Weil functions
and algebraic Theta functions. We deduce  in Section~\ref{sec:bl} a bound for the complexity of computing a basis
of sections for the  bundle associated with a multiple of the natural polarization of $J$. We recall the algebraic
definition of canonical Theta functions in Section~\ref{sec:canotheta} and bound the  complexity of evaluating
such a function at a given point in $J$. Section~\ref{sec:quo} bounds the complexity of evaluating
functions on the quotient of $J$ by  a maximal isotropic subgroup $V$ in  $J[l]$ when $l$ is
an odd prime different from the characteristic of $\KK$. Specific algorithms for genus two curves are given in Section~\ref{sec:g2}.
A complete example is treated in Section~\ref{sec:exa}.

{\bf Context} The algorithmic aspect of isogenies was explorated by Vélu \cite{velu1,velu2} in the context of elliptic
curves. He exhibits  bases of linear spaces made of  Weil functions, then finds invariant functions
using traces. Vélu considers the problem of computing the quotient variety once given some finite subgroup.
The  problem of computing (subgroups of) torsion points is independent and was solved in a somewhat optimal
way by  Elkies \cite{elk}  in the genus one case, using modular equations. It is unlikely that modular equations
will be of any use to accelerate  the computation of torsion points for higher genera, since they all are far too big.
Torsion points may be computed by brute force (torsion polynomials), using the Zeta function when it is known \cite{Couveignesmodp},
or because they come naturally as part of the input (modular curves). We shall not consider this problem and will concentrate on
the computation of the isogeny, once given its kernel. The genus one case  has been surveyed by Schoof \cite{schoof} and Lercier-Morain \cite{LM98}.
The genus two case was studied by Dolgachev and Lehavi \cite{dolga}, and Smith
\cite{smith},  who provide a very elegant geometric  description. However,
the complexity of the resulting algorithm is not given (and is not 
quasi-linear in the degree  anyway). Lubicz and Robert \cite{lr,lr2} provide general
methods for quotienting abelian varieties (not necessarily Jacobians) by maximal isotropic subgroups in the $l$-torsion. Their
method has quasi-linear  complexity $l^{g(1+o(1))}$ when $l$ is a sum of two squares. Otherwise it has complexity $l^{g(2+o(1))}$.
The case of dimension two is treated by Cosset and Robert \cite{cos}. They reach
complexity $l^{2+o(1)}$ when $l$ is the sum of two squares and  $l^{4+o(1)}$ otherwise.
However, the input and mainly the output of these methods is quite different from ours. In the dimension two case,  we can, and must provide a curve $D$ of which $J/V$
is the Jacobian, and an explicit  map  from $C$ into the symmetric square of $D$. We achieve this goal in quasi-linear  time $l^{2+o(1)}$
for every odd prime $l\not = p$.

{\bf Aknowledgements} We thank Damien Robert for  his comments on an early version of this work and Qinq Liu for interesting discussions about holomorphic
differentials.
Tony Ezome is supported by the Simons Foundation via the PRMAIS project.
Jean-Marc Couveignes  is supported by the 
``Agence Nationale de la Recherche'' (project PEACE) and by the cluster of excellence CPU (Numerical certification and reliability).
Experiments presented in this paper were carried out using PARI/GP \cite{PARI2}
and the 
PLAFRIM experimental testbed, being developed under the Inria PlaFRIM 
development action with support from LABRI and IMB and other entities: Conseil Régional d'Aquitaine, Université de Bordeaux and CNRS.

\section{Functions on  Jacobians}\label{sec:fj}

Constructing functions on abelian varieties using  zero-cycles and  divisors is classical \cite{weil1,weil2}. 
In this section, we bound the complexity of evaluating such functions in the special case of Jacobian varieties.
Possible references for the theory of Jacobian varieties are \cite{weil1,langAV,MilneAV,arbarello}.   

Section~\ref{sec:nono} sets some notation about Jacobian varieties.
Section~\ref{sec:alpha} is concerned with  a special case
of Eta functions : those associated to a function on the curve
itself. These functions can be  easily evaluated. 
Section~\ref{sec:aco} recalls well known but important algorithmic
results about curves and  Jacobians.
These algorithmic considerations are of particular interest when
the base field $\KK$ is finite. Bounds on the number of points
on varieties are useful in this context. We recall in 
Section~\ref{sec:diese} a  simple estimate that will suffice
for our purpose.
We provide in Section~\ref{sec:deter} a formula for the divisor
of certain functions on $J$ defined using determinants. We deduce
an expression for Eta functions as combinations of these determinants.
The resulting  algorithm for evaluating  Eta functions
is detailed in Section~\ref{sec:evaleta}.

\subsection{Notation}\label{sec:nono}

We let $\KK$ be a field. Let  $\barKK$ be an algebraic closure of $\KK$. 
If $X$ is a $\KK$-scheme and  if $\LL$ is  an extension of  $\KK$,  we denote by $X_\LL$ the base change
$X\otimes_\KK\LL$ and by $X(\LL)$ the set of $\LL$-points on it.
Let $C$ be a projective, smooth, absolutely integral
curve over $\KK$.  Let $g$ be the genus of $C$. We assume that $g\ge 2$  and we denote by    
$\Pic(C)$ the Picard scheme  of $C$.
For every integer $d$ we denote by 
$\Pic^d(C)$ the component of $\Pic(C)$ representing linear classes of  divisors of degree $d$.
In particular, $J=\Pic^0(C)$ is the Jacobian variety of $C$.
By definition of the Picard scheme,  $\LL$-points on $\Pic(C)$ parameterize linear equivalence classes of divisors   on $C_\LL$. 
We shall make no difference between linear classes of divisors and points  on the Picard scheme. 
The {\it canonical class} on $C$ is denoted $\omega$. It  is represented
by a $\KK$-point on $\Pic^{2g-2}(C)$ which we call $\omega$ also.
If $D$ is a divisor on  $C_\LL$ we denote by $\iota(D)$ 
 its 
linear equivalence class, and 
the corresponding $\LL$-point on $\Pic(C)$.
Let  $u$ be an $\LL$-point on  $\Pic(C)$. We call
\[t_u : \Pic(C)_\LL\rightarrow \Pic(C)_\LL\]  the  translation     by $u$. 
If now  $D$ is a divisor on $\Pic(C)_\LL$ we denote by \[D_u = t_u(D)\]  the translation  of $D$  by $u$.
We call $W \subset \Pic^{g-1}(C)$ the 
algebraic set representing  classes of effective divisors of degree $g-1$.
The pullback $[-1]^*  W \subset  \Pic^{1-g}(C)$
is equal to the translate of $W$ by $-\omega$. We write
\[[-1]^*  W = W_{-\omega}.\]
If there exists a $\KK$-rational point $\theta$ in $\Pic^{g-1}(C)$
such that \[2\theta  = \omega,\]
then we  say that
$\theta$ is a {\it Theta characteristic}. See \cite[Appendix B, \S 3]{arbarello}.
Two such  Theta characteristics differ by a $2$-torsion point in $J$. The
translate  $W_{-\theta}$ is a divisor on  $J$. One has 
\begin{equation}\label{eq:symw}
[-1]^* W_{-\theta}=W_{-\theta}.
\end{equation}
The divisor $W_{-\theta}$ is said to be {\it symmetric}.
We  assume that  we are given   a $\KK$-rational point $O$ on $C$, and
 denote by \[o = \iota (O)\] its linear equivalence class.
This is a $\KK$-point on $\Pic^1(C)$. The   translate $W_{-(g-1)o}$ is a divisor on $J$.
We set \[\kappa  = \omega-2(g-1)o \in J(\KK).\]
We have \[[-1]^*W_{-(g-1)o}=W_{-(g-1)o-\kappa}.\]
We set
\[\thetachar = \theta - (g-1)o \in  J(\KK)\] and check that \[2\thetachar = \kappa.\]
Given  $D$ a divisor on $C$ we write $L(D)$ for the linear space 
$H^0(C,\cO_C(D))$  and 
\[\ell(D)  = \dim ( H^0(C,\cO_C(D)).\]

Let $\ugot = \sum_{1\le i\le I} e_i[u_i]$ be a zero-cycle on $J_\barKK$. 
So $(e_1, e_2, \ldots, e_I) \in \ZZ^I$ and $(u_1, \ldots, u_I) \in J(\barKK)^I$. We set 
\[s(\ugot) = \sum_{1\le i\le I}e_iu_i  \in J (\barKK) \text{ and } \deg(\ugot) = \sum_{1\le i\le I}e_i \in \ZZ.\]
Let $D$ be a divisor on $J_\barKK$.  
The divisor $\sum_{1\le i\le I} e_i D_{u_i}-D_{s(\ugot)}-(\deg(\ugot)-1)D$ is principal.
Let $y$ be a point in $J(\barKK)$ not in the support of this divisor.
Call $\eta_D[\ugot,y]$
the unique  function on $J_\barKK$ having divisor
\[ (\eta_D[\ugot ,y]) = \sum_{1\le i\le I}e_iD_{u_i} -D_{s(\ugot)}-(\deg(\ugot)-1)D\]
and such that 
\[\eta_D[\ugot ,y](y)=1.\]
This definition is  additive in the sense that
\begin{equation}\label{eq:addit}
\eta_D[\ugot+\vgot ,y ] = \eta_D[\ugot ,y] \ttimes \eta_D[\vgot ,y] \ttimes \eta_D[[s(\ugot)]+[s(\vgot)],y]\end{equation}
whenever it makes sense.  If   $D$, $y$,  and $\ugot$ are defined over $\KK$
then $\eta_D \in \KK(J)$.
We write \[\eta_D[\ugot] \in \KK  (J)^*/\KK^*\] when we  consider an Eta function up to a multiplicative scalar.

Equation (\ref{eq:addit})  allows us to evaluate Eta functions by
pieces: we first treat a few special cases and then explain how to combine them to efficiently evaluate any Eta function.
We shall see in Sections~\ref{sec:canotheta} and~\ref{sec:quo} that many interresting functions on $J$ can be expressed as combinations of Eta functions.
In this paper we shall  be firstly  interrested in the special case  $D  = W_{-(g-1)o}$. We  omit the subscript in that case, and write
$\eta[\ugot ,y]$ rather than $\eta_{W_{-(g-1)o}}[\ugot ,y]$.

\subsection{An easy special case}\label{sec:alpha}

Let $f$ be a non-zero  function  in $\KK(C)$.
Following \cite{fd}  one can  naturally  associate  to $f$ a function $\alpha[f]$ in $\KK(J)$ in the following way.
We assume  
that $f$ has degree $d$ and divisor 
\[(f)=\sum_{1\le i\le d} Z_i-\sum_{1\le i\le d} P_i.\]
We call  $z_i=\iota(Z_i)$ (resp.    $p_i = \iota (P_i)$)  the $\barKK$-points in $\Pic^1(C)$
representing the linear equivalence classes of the $Z_i$ (resp.  the $P_i$). 
Let $x$ be a  point in $J(\KK)$ such that  
$x\not\in W_{p_i-go}$ for every $1\le i\le d$.
In particular, $\ell(x+gO)=1$. Indeed every special divisor class of degree $g$ belongs to $W_{\iota(P)}$ for every point $P$ on $C$ since 
the corresponding linear series has positive projective dimension and
we can find a divisor in it containing any given $P$.
Let $D_x$ be the unique effective
divisor of
degree $g$ on $C$ such that $D_x-gO$ belongs to the class $x$.  Write $D_x=D_1+D_2+\dots +D_g$ and set
\begin{equation}\label{eq:alphadef}
\alpha[f](x)=f(D_1)\ttimes f(D_2)\ttimes \, \dots \, \ttimes f(D_g).
\end{equation} The divisor of $\alpha[f]$ is
\begin{equation*}
(\alpha[f])=\sum_{1\le i\le d} W_{z_i-go}-\sum_{1\le i\le d} W_{p_i-go}.
\end{equation*}
Let $y$ be a point in  $J(\KK)$ such that $y\not\in W_{p_i-go}$ and
$y \not\in W_{z_i-go}$ for every $1\le i\le d$.
 Then
\[\alpha[f](x)/\alpha[f](y)=\eta[\sum_{1\le i\le d} [z_i-o]-\sum_{1\le i\le d} [p_i-o]  ,y](x).\]
This provides  an algorithm to evaluate $\eta[\ugot,y]$ 
in the special case when $\ugot$ is a zero-cycle on $J$ with support contained
in $t_{-o}(\iota(C)) \subset J$. 
\subsection{Algorithmic considerations}\label{sec:aco}

Having described in Section~\ref{sec:alpha}  a first method to evaluate Eta functions in some special case, we bound  the complexity
of this method. We take this opportunity to set some notation and convention.

\subsubsection{Convention}

In this text, the
notation $\OO$ stands for a positive absolute constant.
Any statement containing
this symbol becomes true if the symbol is replaced in every occurrence by some 
large enough real number. Similarly,
the notation $\egot(x)$ stands for a real  function of the real
parameter
$x$ alone, belonging to the class $o(1)$.

\subsubsection{Operations in $\KK$} The time needed for one operation
in $\KK$ is  a convenient unit of time.
Let  $\LL$ be  a monogene finite   $\KK$-algebra of  degree $d$.
We will assume that $\LL$ is given as a quotient $\KK[x]/f(x)$ where $f(x)$ is 
a polynomial
in $\KK[x]$.
Every operation in $\LL$ requires  $d^{1+\egot(d)}$
operations in $\KK$. 
When $\KK$ is  a finite field with cardinality $q$, every operation
in $\KK$ requires $(\log q)^{1+\egot(q)}$ elementary operations.

\subsubsection{Operations in $J(\KK)$} We assume that $C$ is given in a reasonable way: for example a plane model with degree polynomial
in the genus $g$. 
Elements in $J(\KK)$ are 
classically represented 
by divisors on $C$. We can also use 
Makdisi's representation \cite{makdisi}  
which is more efficient. For our purpose it will be enough
to know  that one operation in $J(\KK)$ requires 
$g^\OO$ operations  in $\KK$ that is 
$g^\OO\ttimes(\log q)^{1+\egot(q)}$ elementary operations when $\KK$
is a field with $q$ elements.
Given two effective divisors $D$ and $E$ with 
degrees $d$ and $e$ respectively, 
we  are able to compute  a basis of $L(D-E)$ at the expense
of $(gde)^\OO$  operations in $\KK$.   The Brill-Noether algorithm reduces all theses algorithmic
 problems  to the analysis
of the singularities of the given curve. This is classically achieved by blowing up
or using series expansions, but none of these method is fully satisfactory: The complexity of blowing up
is not well understood in the worst cases; and computing series expansions is only possible  when the characteristic
of $\KK$ is zero or large enough. Work by Hess \cite{Hess}, using general  normalization algorithms, provides a
satisfactory algorithm in general. Possible references for these  algorithms  are  Hess \cite{Hess}, Makdisi \cite{makdisi},
Diem \cite{Diem},   or the quick account at the  beginning of  \cite{Couveignesmodp}.

\subsubsection{Evaluating $\alpha[f]$}\label{sec:evalal}

We are given a function $f$ in $\KK(C)$. We are given a  class  $x$ in $J(\KK)$,
represented  by $D_x-gO$ where $D_x$ is effective
with degree $g$. We may see $D_x$ as a zero-dimensional  scheme over $\KK$, and call $\KK[D_x]$ the associated affine
$\KK$-algebra. We assume that $D_x$ does not meet the poles of $f$. Let $P$ be the generic point on $D_x$. Then $f(P)$ belongs to $\KK[D_x]$ and its  norm over
$\KK$ is $\alpha[f](x)$ according to the  definition given in Equation~(\ref{eq:alphadef}). Thus we can  compute  $\alpha[f](x)$ at the expense of   $(gd)^\OO$ operations
in $\KK$, where $g$ is the genus of $C$ and $d$ is the degree of $f$.

\subsection{Number of points on Theta divisors}\label{sec:diese}

We  recall a rough but very general and  convenient upper  bound 
for the number of points in algebraic sets over finite fields.
This bound  was
proved in  \cite[Proposition 12.1]{LG}  by Lachaud and Ghorpade. 
We shall  use it  to estimate the probability of success
of some of the algorithms  presented in this paper.
 
\begin{lemma}[Rough bound for the number of points]\label{lem:LG}
Let $\KK$ be a field with $q$ elements.
Let $X$ be a projective algebraic set  over $\KK$. Let $n$ be the maximum
of the dimensions of the $\KK$-irreducible components of $X$. Let $d$
be the sum of the degrees of the $\KK$-irreducible components of $X$.
Then \[\left|X(\KK)\right|\le d(q^n+q^{n-1}+\dots+q+1).\]
\end{lemma}

Let $\KK$ be a finite field with cardinality $q$
and $C$ a curve over $\KK$ and $O$
a $\KK$\!-rational point on $C$ and 
$J$ the Jacobian of $C$. We assume that the genus $g$ 
of $C$ is at least $2$. Set $\iota(O)=o \in \Pic^1(C)$.  Recall that  $W_{-g(o-1)}$ is  the algebraic
subset of $J$ consisting of all classes $\iota(A-(g-1)O)$ where $A$
is an effective divisor with degree $g-1$.
Let $D$ be an algebraic subset of codimension one in $J$.
We assume  that $D$ is algebraically
equivalent to $kW_{-(g-1)o}$.
Set $l=\max(3,k)$. The divisor  $E=D+(l-k)W_{-(g-1)o}$
is algebraically equivalent 
to $lW_{-(g-1)o}$. After base change to $\bar \KK$ it becomes
linearly equivalent to a translate of $lW_{-(g-1)o}$. Since every translate of 
$W_{-(g-1)o}$
is ample \cite[Chapter II, \S 6]{mumford}  and $l\ge 3$ we deduce 
\cite[Chapter III, \S 17]{mumford} that
$E$ is very ample. We now apply Lemma~\ref{lem:LG} to the hyperplane
section $E$. Its dimension is $n=g-1$ and  its degree $d$
is \[E^{g}=l^{g}\left(W_{-(g-1)o}\right)^{g}=l^g .  g!\] 
according to Poincaré's formula \cite[Chapter I, \S 5]{arbarello}.
So $|D(\KK)|\le |E(\KK)|\le l^g\ttimes (g!)\ttimes
(q^{g-1}+q^{g-2}+\dots+q+1)\le g\ttimes (g !) \ttimes l^g\ttimes q^{g-1}$.
On the other hand, according to \cite[Théorème 2]{lachaud}, the cardinality of $J(\KK)$ is at
least $q^{g-1}(q-1)^2(q+1)^{-1}(g+1)^{-1}$. So the proportion
  $D(\KK)/J(\KK)$ is $\le g^{\OO g}l^g/q$.

\begin{lemma}[Number of points on  divisors]\label{lem:ub}
Let $\KK$ be a finite field with $q$ elements and  $C$ a curve 
of genus $g\ge 2$ over $\KK$.
Let $J$ be the Jacobian of $C$.
 Let $O$ be a $\KK$-point on $C$ and
$o$ the corresponding class in $\Pic^1(C)$.
 Let  $D\subset J$  be an algebraic subset
of  codimension
one,
 algebraically
equivalent to $kW_{-(g-1)o}$ for $k\ge 1$. Set
$l=\max(3,k)$. The number of $\KK$\!-rational points on 
$D$ is bounded from above by $g\ttimes (g !) \ttimes l^g\ttimes q^{g-1}$.
The ratio $|D(\KK)|/ |J(\KK)|$ is bounded from above by $g^{\OO g}l^g/q$.
\end{lemma}

\subsection{Determinants}\label{sec:deter}

The evaluation method presented in Section~\ref{sec:aco} only
applies to  Alpha functions introduced in
Section~\ref{sec:alpha}. These Alpha functions form a subfamily
of Eta functions.
Mascot   introduced in \cite{Mascot} an efficient evaluation 
method that applies
to another interesting  subfamily. 

One   can also define and evaluate 
functions on $J$ using determinants. See
\cite{baker, fay, sb}.
We shall see that  every Eta function can be expressed as a combination
of Alpha functions, as in Section~\ref{sec:alpha},  and determinants. Let $D$ be a divisor on $C$ with degree 
$d\ge 2g-1$. 
Set \[n=\ell(D)=d-g+1.\] Let $ f  =  (f_k)_{1\le k\le n}$ be a basis of $L(D)$.
For $P=(P_l)_{1\le l\le n}$ in $C^n$ disjoint from the positive part of $D$ we set 
\[\beta[f](P)=\det (f_k(P_l))_{k,l}\] and thus define a function
$\beta[f]$ on $C^n$. Call    \[\jmath  : C^n  \rightarrow \Pic^n(C)\] 
the Jacobi integration map. It maps $(P_1, \ldots, P_n)$ onto the class
of $P_1+\dots+P_n$.  We call
\[\pi_l : C^n\rightarrow C\] the projection onto the  $l$-th factor. 
For ${1\le i<j\le n}$  we set 
\[\Delta_{i,j} =\{(P_1, \ldots, P_n) | P_i=P_j\}\subset C^n.\] Let
\[\Delta = \cup_{1\le i<j\le n} \Delta_{i,j} \subset C^n\] be the full diagonal.
The divisor of $\beta[f]$ is 
\begin{equation}\label{eq:fay}
(\beta[f])=\Delta+\jmath ^*(t_{\iota(D)}([-1]^* W))+\sum_{1\le l\le n} \pi_l^*(-D)
\end{equation}
where $t_{\iota(D)}([-1]^*W) = W_{\iota(D)-\omega }\subset \Pic^{n}(C)$ is 
the translate of $[-1]^*W$ by the class of $D$. When $\KK$ has 
characteristic zero Equation~\ref{eq:fay}
results from \cite[Proposition 2.16]{fay}. For general $\KK$,
a Galois theoretic
proof is given  by  Shepherd-Barron in \cite[Corollary 4.2]{sb}. 
Roughly speaking the term  $\Delta$  in Equation~(\ref{eq:fay}) 
means that the determinant vanishes when $P_i=P_j$ because there
are two equal collumns in that case. The $\sum_{1\le l\le n} \pi_l^*(-D)$ says that
poles of the determinant come from poles of the coefficients
in it. The term $\jmath ^*(t_{\iota(D)}([-1]^* W))$ says that if 
the  $n$  points  $P_1$, \dots, $P_n$, are  distinct, 
the determinant vanishes if and only if there exists 
a non-zero function
in $L(D)$ vanishing at $P_1$, \dots, $P_n$.
And this means that $D$ is linearly  equivalent 
to $P_1+\dots+P_n$ plus some effective divisor of degree $g-1$.

We now assume that we have a collection of divisors 
$D= (D^{(i)})_{1\le i\le I}$. We assume that all $D^{(i)}$ have   degree $d=2g-1$. So
$n=\ell(D^{(i)})=g$. We are given a vector of integers
$e = (e_i)_{1\le i\le I}$ such that $\sum_{1\le i\le I} e_i=0$. For every $i$ we choose a basis $f^{(i)} = (f^{(i)}_{k})_{1\le k\le g}$  of
$L(D^{(i)})$.
We assume that  $\sum_{1\le i\le I} e_i \ttimes D^{(i)}$ is the (principal)
divisor of some function $h$ on $C$. We call $\alpha[h]$ the function on $J$ associated with $h$, as constructed in Section~\ref{sec:alpha}.
We set $f=(f^{(i)})_{1\le i\le I} $.   Define  the function
\[\beta[D,e,f]=\prod_{1\le i\le I} \beta[f^{(i)}]^{e_i}\] on $C^g$. It has  divisor
 \[(\beta[D,e,f])=\sum_i e_i \ttimes \jmath ^*(W_{\iota(D^{(i)})-\omega})-\sum_{\stackrel{1\le i\le I}{1\le l\le g}} e_i\ttimes \pi_l^*(D^{(i)}).\]
There exists a function $\beta '[D,e,f]$ on $\Pic^g(C)$ such that  $\beta [D,e,f] =   \beta '[D,e,f] \circ \jmath $.
Indeed, permuting the $g$ points $(P_i)_{1\le i\le g}$ multiplies each factor $\beta[f^{(i)}]$ by the same sign.
We call $\gamma[D,e,f]$ the function on $J=\Pic^0(C)$ obtained by composing   $\beta '[D,e,f]$ with the translation
by $go$. The product    $\gamma[D,e,f]\ttimes \alpha[h]$  has divisor
 \[(\gamma[D,e,f])  +  (\alpha[h]) = \sum_i e_i W_{-(g-1)o+u_i},\]
where \[u_i = \iota(D^{(i)})-\omega-o  \in  J(\KK).\] We deduce that 
\begin{equation}\label{eq:comb}
\gamma[D,e,f]\ttimes \alpha[h] = \eta[\ugot] \in \KK  (J)^*/\KK^*
\end{equation}
 where  $\ugot = \sum_i e_i[u_i]$. This is exactly what we need. Every  Eta
function decomposes (up to a multiplicative  scalar) as the product of
a certain number of determinants times
some Alpha function, which we know how to compute.  In the next 
Section~\ref{sec:evaleta} we deduce  an algorithm for evaluating
Eta functions.

\subsection{Evaluating Eta functions}\label{sec:evaleta}

We explain how to evaluate Eta functions, using the product decomposition
given in Equation~(\ref{eq:comb}).
We are given  $\ugot = \sum_{1\le i\le I} e_i[u_i]$ 
a zero-cycle on $J$. We assume that $u_i\in J(\KK)$ for $1\le i\le I$.
We can  and will assume without loss of generality
that $\deg(\ugot)=\sum_i e_i = 0$ and $s(\ugot)=\sum_i e_iu_i = 0$.
We are given two classes $x$ and $y$ in $J(\KK)$.
The class $x$ is represented by a divisor $D_x-gO$ where
$D_x$ is effective with degree $g$.
The class $y$ is represented similarly by a divisor $D_y-gO$.
We assume that  neither of $x$ and  $y$  belong to the support of the divisor
$\sum_{1\le i\le I}e_iW_{-(g-1)o+u_i}$. We want to evaluate 
$\eta[\ugot,y](x)$.

The algorithm  goes as follows.

\begin{enumerate}
\item For every $1\le i\le I$, find an effective divisor $D^{(i)}$ of degree
$2g-1$ such that $D^{(i)}$ does neither  meet $D_x$ nor $D_y$, and 
$\iota(D^{(i)})-\omega-o$ is  the class $u_i$. 
\item Find a non-zero function $h$ in $\KK(C)$ with divisor 
$\sum_{1\le i \le I}e_iD^{(i)}$.
\item  For every $1\le i\le I$, compute a basis $f^{(i)}=(f^{(i)}_k)_{1\le k\le g}$
of $L(D^{(i)})$. 
\item Write $D_x = X_1+X_2+\dots+X_g$ and 
 $D_y = Y_1+Y_2+\dots+Y_g$ where  $X_k$ and $Y_k$ are points
in $C(\barKK)$ for $1\le k\le g$. 
 For every $1\le i\le I$, compute 
\[\delta_x^{(i)} = \det (f^{(i)}_k(X_l))_{1\le k,l \le g} \text{ and } \delta_y^{(i)} = \det (f^{(i)}_k(Y_l))_{1\le k,l \le g}.\]
\item Compute $\alpha[h](x)$ and $\alpha[h](y)$.
\item Return 
\[\frac{\alpha[h](x)}{\alpha[h](y)} \ttimes{\prod_{1\le i\le I} (\delta_x^{(i)}/\delta_y^{(i)})^{e_i}}.\]
\end{enumerate}

Note that the product above reflects the product in Equation~(\ref{eq:comb}).
The only difference is that we evaluate at two points $x$ and $y$ to fix 
the  multiplicative constant in $\KK^*$. 
We now precise every step.
In step (1) we assume that the class $u_i$ is given
by a divisor $U_i-gO$ where $U_i$ is effective with degree $g$.
We proceed as in \cite[Lemmata 13.1.7-8-9]{Couveignesmodp}.
We choose a canonical divisor $K$ on $C$
and compute $L(U_i-(g-1)O+K)$. With every non-zero function
$f$ in this linear space is associated a candidate   divisor
\[(f)+U_i-(g-1)O+K\] for $D^{(i)}$. We eliminate the candidates
that meet either $D_x$ or $D_y$. The corresponding functions
$f$ belong to a union of at most $2g$ strict subspaces of 
$L(U_i-(g-1)O+K)$. If the cardinality of $\KK$ is bigger than
$2g$ we find a decent divisor $D^{(i)}$ by solving inequalities.
If $\KK$ is too small, we can  replace $\KK$  by a small extension of it.
In any case, we find some $D^{(i)}$ at the expense of $g^\OO$ operations
in $\KK$.

Step (2) is   effective Riemann-Roch. It requires $(g\ttimes |e| )^\OO$ operations
in the base field,  where \[|e| = \sum_{1\le i\le I}|e_i|\]
is the $\ell^1$-norm. 
Step (3)   is similar to step (2) and requires
 $I\ttimes g^\OO$ operations in $\KK$.
Step (4) requires some care. Brute force calculation
with the $X_k$ and $Y_k$ may not be polynomial time in the genus
because the degree over $\KK$ of the
decomposition field of $D_x$ and $D_y$ may
be very large. 
However, if $\KK$ is finite and if $D_x$ is irreducible over $\KK$, 
then this decomposition
field has degree $g$, which is fine with us. 
In general, we write $D_x = \sum_{1\le l\le L} a_lR_l$ where the $R_l$
are pairwise distinct irreducible divisors and the $a_l$ are positive integers.
We compute a new basis $(\phi_k)_{1\le k\le g}$ for
$L(D^{(i)})$ which is adapted to the decomposition
of $D_x$ in the following sense: we start with a basis
of $L(D^{(i)}  - \sum_{l\ge 2}a_lR_l)$, we continue with a basis
of $L(D^{(i)}  - \sum_{l\ge 3}a_lR_l)/L(D^{(i)}  - \sum_{l\ge 2}a_lR_l)$,
we continue  with a basis
of $L(D^{(i)}  - \sum_{l\ge 4}a_lR_l)/L(D^{(i)}  - \sum_{l\ge 3}a_lR_l)$,
and so on. The matrix  $(\phi^{(i)}_k(X_l))_{1\le k,l \le g}$
is block-triangular, so its determinant is a product
of $L$ determinants (one for each $R_l$). We compute
each of these $L$ determinants by brute force and multiply them together.
We multiply the resulting product 
 by the determinant of the transition matrix
between the two bases.

For step (5) we use the method described in Section~\ref{sec:evalal}.
Step (6) seems trivial, but it hides an ultimate difficulty. If $D_x$
is not simple, then all $\delta_x^{(i)}$ are zero and there appear
artificial indeterminacies in the product $\prod_i (\delta_x^{(i)})^{e_i}$.
We use a deformation to circumvent this difficulty. 
We introduce
a formal parameter $t$ and consider the field $\LL = \KK ((t))$ of formal
series in $t$ with coefficients in $\KK$.
Consider for example the worst case in which 
 $D_x$ is $g$ times a point $A$. We fix a local parameter 
$z_A \in \KK(C)$ at 
$A$. We fix $g$ pairwise distinct scalars $(a_m)_{1\le m \le g}$
in $\KK$. In case the cardinality of $\KK$ is $<g$,  we replace
$\KK$ by a small degree extension of it.
 We denote $X_1(t)$, $X_2(t)$, \ldots, $X_g(t)$, the $g$ points in
$C(\LL)$ associated with the values $a_1t$, \ldots , $a_gt$,
of the local  parameter $z_A$. We perfom the calculations
described above with $D_x$ replaced by $D_x(t) =
X_1(t)+\dots +X_{g}(t)$, and set $t=0$ in the result. Since we use a field of
series, we care about the necessary $t$-adic accuracy. This is
the maximum $t$-adic valuation of the $\beta[f^{(i)}](D_x(t))$.
Assuming that $x$ does not belong to the support 
of the divisor $(\eta[\ugot]) = \sum_{1\le i\le I}e_iW_{-(g-1)o+u_i}$,
these valuations  all are equal to 
$g(g-1)/2$.
So the complexity remains polynomial in the genus $g$.
In case $\KK$ is a finite field we obtain the theorem below.
\begin{theorem}[Evaluating Eta functions on the Jacobian]\label{th:slo}
There exists a deterministic   
algorithm that takes as input 
\begin{itemize}
\item a finite field $\KK$ with cardinality $q$, 
\item a curve $C$ of genus $g\ge 2$ over $\KK$,
\item  a collection
of $\KK$-points  $(u_i)_{1\le i\le I}$ on the Jacobian $J$ of $C$,
\item 
a zero-cycle $\ugot = \sum_{1\le i\le I} e_i[u_i]$
on   $J$,  such that $\deg(\ugot)=0$ and $s(\ugot)=0$,
\item a point $O$ in $C(\KK)$, 
\item and two  points $x, y \in J(\KK)$, not in $\cup_{1\le i\le I}W_{-(g-1)o+u_i}$.
\end{itemize}
The algorithm  computes $\eta[\ugot,y](x)$ in time 
$(g\ttimes |e|)^\OO\ttimes (\log q)^{1+\egot(q)}$, where $|e| = \sum_{1\le i\le I}|e_i|$ is the $\ell^1$-norm of
$e$.
\end{theorem}

Using fast exponentiation and Equation~(\ref{eq:addit})
in the algorithm above, we can  evaluate Eta functions
in time  $g^\OO\ttimes I\ttimes \left(\log |e|\right)  \ttimes (\log q)^{1+\egot(q)}$. However,
this method  may fail when one of the arguments $x$ or $y$
belongs to the
support of the divisor of some intermediate factor.
According to Lemma~\ref{lem:ub} 
the proportion of such $x$ in $J(\KK)$ 
is $\le g^{\OO g} \ttimes I\ttimes \left(\log |e|\right)/q$. A fast  method
that works for a large proportion of the inputs will be enough
to us in the sequel.

\begin{proposition}[Fast evaluation of  Eta functions on the Jacobian]\label{prop:fas}
There exists a deterministic   
algorithm that takes as  input 
\begin{itemize}
\item a finite field $\KK$ with cardinality $q$, 
\item a curve $C$ of genus $g\ge 2$ over $\KK$, 
\item a point $O$ in $C(\KK)$, 
\item a collection
of $\KK$-points  $(u_i)_{1\le i\le I}$ on the Jacobian $J$ of $C$,
\item a zero-cycle $\ugot = \sum_{1\le i\le I} e_i[u_i]$
on   $J$, such that $\deg(\ugot)=0$ and $s(\ugot)=0$,
\item 
and two  points $x, y \in J(\KK)$, not in $\cup_{1\le i\le I}W_{-(g-1)o+u_i}$.
\end{itemize}
The algorithm  returns either {\rm FAIL} or  $\eta[\ugot,y](x)$ in time 
\[g^\OO\ttimes I\ttimes \left( \log |e|\right)  \ttimes (\log q)^{1+\egot(q)},\] where 
$|e| = \sum_{1\le i\le I}|e_i|$ is the $\ell^1$-norm of
$e$. 

For given  $\KK$, $C$, $\ugot$, $O$, there exists a subset 
$\FAIL(\KK,C,\ugot, O)$ 
of $J(\KK)$ with density \[\le g^{\OO g} \ttimes I\ttimes \log (|e|)/q\] 
and such that the algorithm succeeds whenever neither 
 $x$ nor  $y$ belongs to $\FAIL(\KK,C,\ugot,O)$.
\end{proposition}

Fast exponentiation for evaluating
Weil functions on abelian varieties first  appears in work by Miller \cite{miller} in the context of pairing computation on elliptic curves.

\section{Bases of linear spaces}\label{sec:bl}

Being able to evaluate   Eta functions $\eta[\ugot , y]$
we now consider an integer 
 $l\ge 2$ and 
look for  a  basis of  $H^0(J,\cO_{J}(lW_{-(g-1)o}))$.
A related problem is to pick random functions in this linear space 
with close enough to uniform probability.
We assume that  the base field 
is finite, and  use the rough consequences of
Weil bounds stated   in Section~\ref{sec:diese}.
Fix two positive coprime  integers $a$ and $b$ such that $a+b=l$. 
For every $u$
and $y$ in $J(\KK)$ such that 
$y\not\in W_{-(g-1)o}\cup W_{-(g-1)o+au}\cup W_{-(g-1)o-bu}$
 call  $\tau[u,y]$ the unique function with divisor
\[(\tau[u,y])=bW_{-(g-1)o + au}+aW_{-(g-1)o-bu}-lW_{-(g-1)o}\]
such that $\tau[u,y](y)=1$. So 
\[\tau[u,y] = \eta[b[au]+a[-bu],y].\]

Let $\tau[u]$ be the class of $\tau[u,y]$ in $\KK  (J)^*/\KK^*$. 
When $u$ is an $l$-torsion point $\tau[u]=\eta[l[au]]$ is a level $l$ Theta function.  It is a classical result of the theory of Theta functions
that the collection of all $\eta[l[u]]$ when $u$ runs over $J[l](\barKK)$ generate $H^0(J_\barKK,\cO_{J_\barKK}(lW_{-(g-1)o}))$. See  \cite[Theorem 3.2.7]{birk} in case $\KK$
has characteristic zero and \cite[\S 10]{tata3} in general, or  Section~\ref{sec:canotheta} below.  So the collection
of all $\tau[u]$ when  $u$ runs over the set
$J[l](\barKK)$ is a generating set for  $\PP(H^0(J_\barKK,\cO_{J_\barKK}(lW_{-(g-1)o})))$.
So the map
$u\mapsto  \tau[u]$ from $J$ to  $\PP(H^0(J,\cO_{J}(lW_{-(g-1)o})))$
is non-degenerate.
Hyperplane sections  for this map are algebraically equivalent
to $ablW_{-(g-1)o}$.

We  pick a random element $u$ in $J(\KK)$,  using the  Monte Carlo probabilistic
algorithm given in \cite[Lemma 13.2.4]{Couveignesmodp}.
This algorithm returns 
a random element $u$ with uniform probability inside  a subgroup
of $J(\KK)$ with index $\xi \le \OO g^\OO$.
We then consider the function 
$\tau[u,y]$ where  $y$ is any point in $J(\KK)$ not in 
$W_{-(g-1)o}\cup  W_{-(g-1)o+au}\cup W_{-(g-1)o-bu}$.
According to Lemma~\ref{lem:ub},
for every hyperplane $H$ in  $\PP(H^0(J,\cO_{J}(lW_{-(g-1)o})))$, the 
proportion  of $u\in J(\KK)$ such that 
 $\tau[u]$ belongs to $H$ is $\le (lg)^{\OO g}/q$. We assume that
$q$ is large enough to make this proportion
smaller than $\le 1/(2\xi)$. The probability that 
$\tau[u]$ belongs to $H$ is then $\le 1/2$.

\begin{proposition}[Random functions]\label{prop:randomtheta}
There exists a constant $\OO$ such that the following is true. There exists
a probabilistic Las Vegas algorithm that takes as input 
\begin{itemize}
\item three 
integers $l\ge 2$, $a\ge 1$, and $b\ge 1$, such that $a$ and $b$ are coprime and  $a+b=l$,   
\item a  curve $C$
of genus $g\ge 2$ 
over a field $\KK$ with $q$ elements, such that $q\ge (lg)^{\OO g}$,
\item  a point $O$ in $C(\KK)$.
\end{itemize}
The algorithm  returns
a pair  $(u,y)$ in $J(\KK)^2$ such that $\eta[u,y]\in   H^0(J,\cO_{J}(lW_{-(g-1)o}))$ is a random function with probability  measure $\mu$ such that 
$\mu(H)\le 1/2$ for every hyperplane $H$ in $H^0(J,\cO_{J}(lW_{-(g-1)o}))$.
The algorithm
runs in time $g^\OO \ttimes \left(\log l\right)\ttimes (\log q)^{1+\egot(q)}$.
\end{proposition}

In order to find  a basis of $H^0(J,\cO_{J}(lW_{-(g-1)o}))$ we take  $I\ge \OO \ttimes l^g\ttimes \log(l^g)$ and pick $I$
random elements $(u_i)_{1\le i\le I}$ in $J(\KK)$ as explained above. 
For every $i$ we find a $y_i$ in $J(\KK)$ such that 
$y_i\not\in W_{-(g-1)o} \cup W_{-(g-1)+au_i}\cup W_{-(g-1)o-bu_i}$.
We pick another $I$ elements $(w_j)_{1\le j\le I}$ 
such that $w_j\not\in W_{-(g-1)o}$. We compute
$\tau[u_i,y_i](w_j)$ for every pair $(i,j)$. We put the corresponding
$I\times I$ matrix  in echelon form. If the rank is $l^g$ we deduce
a basis for both $H^0(J,\cO_{J}(lW_{-(g-1)o}))$ and its dual all at a time. 
 
\begin{proposition}[Basis of $H^0({J},\cO_{J}(lW_{-(g-1)o}))$]\label{prop:randomtheta2}
There exists a constant $\OO$ such that the following is true. There exists
a probabilistic Las Vegas algorithm that takes as input 
\begin{itemize}
\item three 
integers $l\ge 2$, $a\ge 1$, and $b\ge 1$, such that $a$ and $b$ are coprime and $a+b=l$,   
\item a  curve $C$
of genus $g\ge 2$ 
over a field $\KK$ with $q$ elements,  such that $q\ge (lg)^{\OO g}$,
\item 
a point $O$ in $C(\KK)$.
\end{itemize}
The algorithm returns $l^g$
triples  $(u_i,y_i,w_i)\in {J}(\KK)^3$ such that 
 $(\tau[u_i,y_i])_{1\le i\le l^g}$ is  a basis of $H^0({J},\cO_{J}(lW_{-(g-1)o}))$
and $(w_i)_{1\le i\le l^g}$  is a basis of its dual.
The algorithm
runs in time \[g^\OO \ttimes  (l^{ g})^{\omega(1+\egot(l^g))}\ttimes 
\left( \log q\right)^{1+\egot(q)}\]
where $\omega $ is the exponent in matrix multiplication.
\end{proposition}

One finds in \cite[Chapter 15]{shok} an elegant presentation of the complexity of matrix multiplication, a definition
of the exponent $\omega$, and a reasonably simple proof of Coppersmith and Winograd's inequality   $\omega < 2.41$. It is an open question whether $\omega =2$.
The current best result in this direction is the proof by  Le Gall in \cite{legall}  that $\omega <2.3728639$.  

If the condition $q\ge (lg)^{\OO g}$ in the Proposition above  is not met, we work with
a small extension $\LL$ of $\KK$, then make a descent from
$\LL$ to $\KK$ on the result. The resulting basis will consist of
traces of Tau functions.

\section{Canonical Theta functions}\label{sec:canotheta}

Let $l\ge 3$ be an odd prime. We assume that $l$ is different
from the characteristic $p$ of $\KK$. 
According to Equation~(\ref{eq:symw})  the divisor
$W_{-\theta} \subset J$ is symmetric. Let 
$\cL=\cO_{J}(lW_{-\theta})$  be the sheaf  associated 
to the divisor $lW_{-\theta}$. The Theta group
$\cG(\cL)$ fits in the exact sequence
\[1\rightarrow \Gm \rightarrow \cG(\cL)\rightarrow {J}[l]\rightarrow 0.\]
In this section we recall the definition of algebraic Theta functions. 
Level $l$ Theta functions belong to 
$H^0({J_\barKK},\cO_{J_\barKK}(lW_{-\theta}))$ and they
generate it. 
They  are  useful to define descent data. We shall need 
them in Section~\ref{sec:quo}.
 In this section we bound the complexity of evaluating Theta functions.

\subsection{Defining canonical Theta functions}
We recall the properties of canonical Theta
functions as defined e.g. in \cite[3.2]{birk} or 
\cite[\S 3]{tata3}. We shall see
that canonical Theta functions can be
characterized more easily when the level $l$ is odd.
For   $u$ in ${J}[l](\barKK)$ we let $\theta_u$ be a  function on ${J}_\barKK$
with divisor
$l(W_{-\theta+u}-W_{-\theta})$.  
We call \[\aa_u : H^0({J_\barKK},\cO_{J_\barKK}(lW_{-\theta})) \rightarrow 
H^0({J_\barKK},\cO_{J_\barKK}(lW_{-\theta}))\]
the endomorphism that maps every function $f$ onto the product of $\theta_u$ by $f\circ t_{-u}$.
For the moment $\theta_u$ and  $\aa_u$ are only defined up
to a multiplicative  scalar.  We now  normalize both of them.
We want the $l$-th iterate of $\aa_u$ to  be the identity.
So $\theta_u \ttimes  \left( \theta_u\circ t_u\right) \ttimes \, \dots \, \ttimes
\left(\theta_u\circ t_{(l-1)u}\right)$ should be one. 
We therefore divide $\theta_u$ by one of the
$l$-th roots of the above product
to ensure that 
$\aa_u$ 
has order dividing $l$. Now $\theta_u$ and 
$\aa_u$  are defined up to an $l$-th root of unity.
We compare $[-1]\circ \aa_u \circ [-1]$ and $\aa_u^{-1}$. They
differ by an $l$-th root of unity $\zeta$. 
Since $l$
is odd, $\zeta$ has  square root $\zeta^{(l+1)/2}$. Dividing 
$\aa_u$ and $\theta_u$ by this square root we complete
their definition. 

\begin{proposition}[Canonical Theta functions]
For every $u$ in ${J}[l](\barKK)$ there is a unique function 
$\theta_u$ with divisor $l(W_{-\theta+u}-W_{-\theta})$ such that 
\begin{equation}\label{eq:prod1}
\theta_u \ttimes \left( \theta_u\circ t_u\right) \ttimes
\left( \theta_u\circ t_{2u}\right) \ttimes \dots \ttimes
\left(\theta_u\circ t_{(l-1)u}\right)=1
\end{equation}
 and 
\begin{equation}\label{eq:sym}
\theta_u\circ [-1] =\left(\theta_u\circ t_u\right)^{-1}.
\end{equation}
Further $\theta_{-u}=\theta_u\circ [-1]$.
The map $u\mapsto \theta_u$ is
Galois equivariant: for every $\sigma$ in the absolute Galois group
of $\KK$ we have \[{}^\sigma \! \theta_u=\theta_{\sigma(u)}.\]

Let $\aa_u$ be the endomorphism 
\begin{equation*}
\xymatrix{
\aa_u : &  H^0({J_\barKK},\cO_{J_\barKK}(lW_{-\theta})) \ar@{->}[r]    &  H^0({J_\barKK},\cO_{J_\barKK}(lW_{-\theta}))\\
&f\ar@{|->}[r]& \theta_u\ttimes
\left(f\circ t_{-u}\right).
}
\end{equation*}
we have  $\aa_u^l=1$ and $[-1]\circ \aa_u \circ [-1]=
\aa_{-u}=\aa_u^{-1}$.
The map $u\mapsto \aa_u$ is
Galois equivariant.
\end{proposition}
\begin{proof}
There only remains to prove the  equivariance property.
It follows from the equivariance of conditions~(\ref{eq:prod1}) and~(\ref{eq:sym}).
\end{proof}

For $u$ and $v$ in ${J}[l](\barKK)$ we write 
\[e_l(u,v)=\aa_u\aa_v\aa_u^{-1}\aa_v^{-1}\in \mu_l\]
for the commutator
pairing and 
\[f_l(u,v)=\sqrt{e_l(u,v)}=\left(e_l(u,v)\right)^{\frac{l+1}{2}}\]
for the half pairing. We check that
\begin{equation}\label{eq:addtheta}
\theta_{u+v}=f_l(u,v).\theta_v\ttimes \left(\theta_u\circ t_{-v}\right)=
f_l(v,u).\theta_u\ttimes \left(\theta_v\circ t_{-u}\right),
\end{equation} and
\[\aa_{u+v}=f_l(u,v).\aa_v\aa_u=f_l(v,u).\aa_u\aa_v,\]
and \[\aa_u(\theta_v)=f_l(u,v).\theta_{u+v}.\]

\subsection{Evaluating  canonical Theta functions}

We relate the canonical Theta functions to the
 Eta functions introduced in Section~\ref{sec:fj}
and show how  to evaluate  them.
We assume that we are given $u$ and $x$ in ${J}(\KK)$ with
$lu=0$, and we want to evaluate
$\theta_u(x)$. 
We assume that $x\not\in W_{-\theta}$.
Since $l$ is odd we set \[v=\frac{l+1}{2}\ttimes u \in {J}(\KK).\]
We deduce from Equation~(\ref{eq:addtheta}) that 
\[\theta_u(x)=\theta_{v}(x)\ttimes \theta_{v}(x-v)\]
provided that $x\not\in W_{-\theta+v}$.
On the other hand, we deduce from Equation~(\ref{eq:sym})
that 
\[\theta_{v}(x)\ttimes \theta_{v}(v-x)=1\]
provided that $x\not\in W_{-\theta}\cup W_{-\theta+v}$.
So 
\[\theta_u(x)= \theta_v(x-v)/\theta_v(v-x)\]
provided that $x\not\in W_{-\theta}\cup W_{-\theta+v}$.
Since $\theta_v$ and $\eta[l[v]]\circ t_{\thetachar}$ have the same divisor
we deduce that 
\begin{equation}\label{eq:thethe}
  \theta_u(x)= \eta[l[v],v-x+\thetachar](x-v+\thetachar)
\end{equation}
provided that $x\not\in W_{-\theta}\cup W_{-\theta+v}$.

Thanks to   Equation~(\ref{eq:thethe}),  evaluating a canonical Theta function $\theta_u(x)$
reduces to the evaluation of one  Eta function. This can be  done  as explained in
Section~\ref{sec:evaleta}.   Applying Theorem~\ref{th:slo} we find that
the computational cost is bounded from above by 
$(gl)^\OO\ttimes (\log q)^{1+\egot(q)}$.

\begin{proposition}[Evaluating canonical Theta functions]\label{prop:cano}
There exists a deterministic  
algorithm that takes as input 
\begin{itemize}
\item a finite field $\KK$ with characteristic $p$ and cardinality $q$, 
\item a curve $C$ of genus $g\ge 2$ over $\KK$, 
\item a Theta characteristic
$\theta$ defined over $\KK$, 
\item an odd prime integer $l\not =  p$, 
\item and two points $u$ and   $x$  in   ${J}(\KK)$ such that $lu=0$,
and 
\[x \not\in W_{-\theta}  \cup W_{-\theta+v},\]
 where 
\[v=\frac{l+1}{2}\ttimes u \in {J}(\KK).\]
\end{itemize}
The algorithm computes   $\theta_u(x)$
in time $(gl)^\OO\ttimes (\log q)^{1+\egot(q)}$.  
\end{proposition}

According to Proposition~\ref{prop:fas}  we can accelerate 
the computation using fast exponentiation.
The resulting  algorithm will  fail when the argument $x$ belongs to the
support of the divisor of some intermediate factor.

\begin{proposition}[Fast evaluation  of canonical Theta functions]\label{prop:cano2}
There exists a deterministic  
algorithm that takes as input
\begin{itemize}
\item  a finite field $\KK$ with cardinality $q$ and characteristic $p$, 
\item a curve $C$ of genus $g\ge 2$ over $\KK$, 
\item a Theta characteristic
$\theta$ defined over $\KK$, 
\item an odd prime  integer $l\not = p$, 
\item and two points $u$ and   $x$  in   ${J}(\KK)$ such that $lu=0$.
\end{itemize}
The algorithm returns either {\rm FAIL} or   $\theta_u(x)$
in time $g^\OO\ttimes (\log q)^{1+\egot(q)}\ttimes \log l$.  
For given  $\KK$, $C$, $\theta$, $u$, there exists a subset 
$\FAIL(\KK,C,\theta,u)$ 
of $J(\KK)$ with density
$\le g^{\OO g} \ttimes \left(\log l\right)/q$  and such that 
the algorithm succeeds whenever 
 $x$ does not  belong to $\FAIL(\KK,C,\theta, u)$.  
\end{proposition}

\section{Quotients of Jacobians}\label{sec:quo}

Let $V\subset {J}[l]$ be a 
maximal isotropic subgroup  for the commutator pairing, let 
$f : {J}\rightarrow J/V$
be the quotient map.  Let $\cL=\cO_{J}(lW_{-\theta})$.
The map $v\mapsto \aa_v$ is a 
homomorphism $V\rightarrow \cG(\cL)$ lifting the inclusion
$V\subset {J}[l]$.  
This canonical lift provides a descent datum for $\cL$ onto $J/V$. We call 
$\cM$ the corresponding
sheaf  on ${J}/V$. This is a symmetric principal polarization. In particular,
$h^0(\cM)=1$ and there is  a unique effective divisor $Y$
on $J/V$ associated with $\cM$. We set $X=f^*Y$. This is an effective
divisor linearly equivalent to $lW_{-\theta}$ and invariant by $V$.
Let $\ugot = \sum_{1\le i\le I} e_i[u_i]$ be a zero-cycle in ${J}$.
Let $y$ be a point on ${J}$. We assume that $y$ does not belong
to the support of the divisor
$\sum_{1\le i\le I}e_iX_{u_i} -X_{s(\ugot)}-(\deg(\ugot)-1)X$.
Recall that $\eta_X[\ugot,y]$
is the unique  function on ${J}$ having divisor
\[ (\eta_X[\ugot ,y]) = \sum_{1\le i\le I}e_iX_{u_i} -X_{s(\ugot)}-(\deg(\ugot)-1)X\]
and such that 
\[\eta_X[\ugot ,y](y)=1.\]
Set $v_i=f(u_i)\in {J}/V$ for every $1\le i\le I$ and let $\vgot = f(\ugot)= \sum_{1\le i\le I} e_i[v_i]$ be the image of $\ugot$ in
the group  of zero-cycles on $J/V$.
There is a function  with divisor $\sum_{1\le i\le I} e_i Y_{v_i}-Y_{s(\vgot)}-(\deg(\vgot)-1)Y$ on  ${J}/V$.
Composing this function
with $f$ we obtain a function on ${J}$ having the same divisor as 
 $\eta_X[\ugot,y]$. So  $\eta_X[\ugot,y]$ is invariant
by $V$ and can be identified with the unique 
 function on $J/V$ with divisor
 $\sum_{1\le i\le I} e_i Y_{v_i}-Y_{s(\vgot)}-(\deg(\vgot)-1)Y$, and 
 taking value $1$ at $f(y)$. 
When dealing with the quotient ${J}/V$ it will be useful to represent a point $z$ on 
$J/V$ by a point $x$ on ${J}$ such that $f(x)=z$. Such an $x$ is in turn represented by a divisor $D_x-gO$ on $C$.
It is then natural to evaluate functions like $\eta_X[\ugot,y]$ at such an $x$. For example, taking 
$\ugot = m[u]$ for $u$ an $m$-torsion point,  the function  $\eta_X[\ugot,y]$ is essentially a Theta function
of level $m$ for the quotient ${J}/V$. Evaluating such functions at a few points, we find projective  equations
for $J/V$. This will show very useful in Section~\ref{sec:g2}.
Section~\ref{sec:ed} provides an expression of
$\eta_X[\ugot,y]$ as a product involving a function $\Phi_V$
defined as an eigenvalue for  the canonical
lift of $V$ in $\cG(\cL)$. The complexity
of evaluating $\Phi_V$ is bounded in  Section~\ref{sec:trace}.

\subsection{Explicit descent}\label{sec:ed}

We look for  a function   
$\Phi_V$ with divisor $X-lW_{-\theta}$ on ${J}$.
Let $V^D=\Hom(V,\GG_m)$ be the dual of $V$. 
For every character
$\chi$ in $V^D$ we denote  $H_\chi$ the $1$-dimensional
subspace of
$H^0({J},\cO_{J}(lW_{-\theta}))$  where $V$ acts through multiplication by  $\chi$.
Then \[\aa_V=\sum_{v\in V}\aa_v\] 
is a surjection from $H^0({J},\cO_J(lW_{-\theta}))$ onto $H_1$.
We pick a random function  in $H^0({J},\cO_J(lW_{-\theta}))$ as explained in 
Proposition~\ref{prop:randomtheta}, and apply    $\aa_V$ to
it. 
With probability $\ge 1/2$ the resulting function  is a non-zero function
in $H_1$. We call  this  function $\Phi_V$.
We will explain in Section~\ref{sec:trace} how to evaluate $\Phi_V$ at a given point on ${J}$.
We now explain how to express any $\eta_X[\ugot]$ as a multiplicative
combination of  $\Phi_V$ and its translates. Without loss of generality we can assume
that $s(\ugot)=0$ and $\deg(\ugot)=0$.
We assume that  $y\not\in \bigcup_i W_{-\theta+u_i} \cup  \bigcup_iX_{u_i}$.
The composition $\Phi_V \circ t_{-u_i}$ has divisor
$X_{u_i}-lW_{-\theta+u_i}$. The composition $\eta[\ugot, y+\thetachar] \circ t_{\thetachar}$ has divisor
$\sum_ie_iW_{-\theta+u_i}$. So 
\begin{equation*}
\eta_X[\ugot,y](x)=\left(\eta[\ugot, y+\thetachar](x+\thetachar) \right)^l \ttimes \prod_{1\le i\le I}\left( \Phi_V (x-u_i)\right)^{e_i}
\ttimes \prod_{1\le i\le I}\left( \Phi_V (y-u_i)\right)^{-e_i}.
\end{equation*}

\subsection{Evaluating functions on ${J}/V$}~\label{sec:trace}

We now bound the cost of evaluating $\Phi_V$ at a given point $x\in {J}(\KK)$. We assume that $l$ is odd and prime to the characteristic $p$ of $\KK$.
We are given
two coprime integers $a$ and $b$ such that $a+b=l$, and two elements $u$ and $y$ in ${J}(\KK)$ such that
$y\not \in W_{-\theta}\cup W_{-\theta+au}\cup W_{-\theta-bu}$. 
The function $\Phi_V$ is the image by $\aa_V$ of 
some  function $\tau$ in $H^0({J},\cO_{J}(lW_{-\theta}))$.  We choose 
$\tau$ to be  the function  
\[\tau =  \tau[u,y+\thetachar]\circ t_\thetachar=\eta[b[au]+a[-bu],y+\thetachar]\circ t_\thetachar.\]
The  $\KK$-scheme $V$ is given by a 
collection of field extensions $(\LL_i/\KK)_{1\le i\le I}$ and a point 
$w_i \in V(\LL_i)$ for every $i$ such that
$V$ is the disjoint union of the $\KK$-Zariski closures of all $w_i$. In particular,
$\sum_id_i=l^g$  where  $d_i$ is the degree of $\LL_i/\KK$ and the $\LL_i$ are the minimum
fields of definition for the $w_i$. 
Equivalently, we may be given a separable algebra $\LL = \KK[V]$ of degree $l^g$ over $\KK$ and a 
point \[\bw\in V(\LL)\subset {J}(\LL).\]
We are given an  element
$x$ in ${J}(\KK)$ such that $x\not\in \cup_{w\in V}W_{-\theta+w}$.
The value \[\aa_\bw (\tau )(x) = \theta_\bw (x)\ttimes \tau (x-\bw  ) = \theta_\bw (x)\ttimes \eta [b[au]+a[-bu],y+\thetachar](x-\bw+\thetachar)\]
 of  $\aa_\bw(\tau)$ at $x$
is an element of the affine algebra $\KK[V]$. Its trace over $\KK$ is equal to
$\Phi_V(x)$.

\begin{theorem}[Evaluating functions on quotients
${J}/V$]\label{prop:fquo}
There exists a deterministic  
algorithm that takes as input 
\begin{itemize}
\item a finite field $\KK$ with characteristic
$p$ and  cardinality $q$, 
\item 
a curve $C$ of genus $g\ge 2$ over $\KK$, 
\item a zero-cycle $\ugot=\sum_{1\le i\le I}e_i[u_i]$ 
on the Jacobian ${J}$ of $C$
such that $u_i\in J(\KK)$ for every $1\le i\le I$, $\deg (\ugot)=0$,
and $s(\ugot)=0$,
\item a Theta characteristic
$\theta$ defined over $\KK$, 
\item an odd prime  integer $l\not =p$, 
\item a maximal isotropic $\KK$-subgroup scheme $V\subset {J}[l]$,
\item two classes $x$ and  $y$ in ${J}(\KK)$ such that
$y\not\in \bigcup_i W_{-\theta+u_i} \cup  \bigcup_iX_{u_i}$.
\end{itemize}
The algorithm returns
either {\rm FAIL} or    $\eta_X[\ugot,y](x)$
in time $I\ttimes \left(\log |e|\right)\ttimes g^\OO\ttimes (\log q)^{1+\egot(q)}\ttimes l^{g(1+\egot(l^g))}$, where $|e| = \sum_{1\le i\le I}|e_i|$ is the $\ell^1$-norm of
$e$.  
For given  $\KK$, $C$, $\ugot$, $\theta$, $V$
 there exists a subset  $\FAIL(\KK,C,\ugot, \theta,V)$ of
$J(\KK)$ 
with density $\le I\ttimes \left(\log |e|\right)\ttimes g^{\OO g} 
\ttimes l^{g^2}\ttimes \left(\log l\right)/q$  
and such that the algorithm succeeds whenever none
of $x$ and $y$ belongs to $\FAIL(\KK,C,\ugot,\theta,V)$.
\end{theorem}

\section{Curves of genus two}\label{sec:g2}

In this section we assume that the characteristic  $p$ of $\KK$
is odd. We bound the complexity of computing an isogeny $J_C\rightarrow J_D$ 
between two Jacobians of dimension two.
We give in Section~\ref{sec:alfo} the expected
form of  such an  isogeny. In Section~\ref{sec:ds} we  characterize 
the isogeny  as the solution
of some system of  differential equations. As a consequence of these differential equations
we can compute such an isogeny in two steps:
we first compute the image of a $(\KK[t]/t^3)$-point on $C$
by the isogeny,  then
lift to $\KK[[t]]$. We explain in Section~\ref{sec:ci} how
to compute images of points.
The main result in this section is Theorem~\ref{th:ci} below.

\subsection{Algebraic form of the isogeny}\label{sec:alfo}

Let $C$ be a  projective, smooth,  absolutely integral
curve of genus $2$ over $\KK$.
We assume that $C$ is given by the affine singular model
\begin{equation}\label{eq:g2}
v^2 = h_C(u)
\end{equation} where $h_C$ is a polynomial of degree $5$.
Let $O_C$ be the unique place at infinity.
Let $J_C$ be the Jacobian of $C$ and let 
$j_C : C\rightarrow J_C$ be the Jacobi map
with origin $O_C$.  The image of a point $P$ on $C$
by $j_C$ is the class of $P-O_C$.
Let  $D$ be another projective, smooth, absolutely irreducible
curve of genus $2$ over $\KK$.
We assume that  $D$ is given by the affine singular
model  $y^2 = h_D(x)$ where $h_D$ is a polynomial of degree $5$ or
$6$. Let $K_D$ be a canonical divisor on $D$.
Call $D^{(2)}$ the symmetric square
of $D$ and let $j_D^{(2)} : D^{(2)}\rightarrow J_D$
be the map sending the pair $\{Q_1,Q_2\}$ onto
the class $z =  j_D^{(2)}  (\{Q_1,Q_2\})$ 
of $Q_1+Q_2-K_D$. This is a birational morphism. We define the Mumford
coordinates
\begin{eqnarray*}
\bs (z)&=& x(Q_1)+x(Q_2) ,\\
\bp (z)&=&  x(Q_1)\ttimes \, x(Q_2),\\
\bq (z)&=& y(Q_1)\ttimes \, y(Q_2),\\
\br (z) &=&(y(Q_2)-y(Q_1))/(x(Q_2)-x(Q_1)).
\end{eqnarray*}
The function field of $J_D$ is $\KK(\bs,\bp,\bq,\br)$.
The function field of the Kummer  variety of $D$ is
$\KK(\bs,\bp,\bq)$.
We assume that there exists an isogeny  $f : J_C\rightarrow J_D$ 
with kernel $V$, a maximal isotropic group in $J_C[l]$,  where $l$ is
an odd prime different from the characteristic $p$ of $\KK$.
We define $F : C \rightarrow J_D$ to be the composite
map $f\circ j_C$. 
The exists a  unique morphism $G : C\rightarrow D^{(2)}$ such that the following diagram commutes. 
\begin{equation*}
\xymatrix{
 &  D^{(2)} \ar@{->}^{j_D^{(2)}}[dd]\\
C \ar@{->}^G[ur] \ar@{->}_F[dr]&\\
&J_D.
}
\end{equation*}
For every point $P  = (u,v)$  on $C$
we have $F((u,-v))=-F(P)$. We deduce the following algebraic 
description of the map $F$
\begin{eqnarray}\label{eq:defisog}
\bs (F(P))&=& \bS(u),\\\nonumber
\bp (F(P))&=&  \bP(u),\\\nonumber
\bq (F(P))&=& \bQ(u),\\\nonumber
\br (F(P)) &=&v\bR(u),\nonumber
\end{eqnarray}
where $\bS$, $\bP$, $\bQ$, $\bR$ are rational fractions
in one variable. 
Let $O_D$ be a point on $D$.
Let $Z$ be the algebraic subset of
$D^{(2)}$ consisting of pairs $\{O_D,Q\}$ for some $Q$ in $D$. Let $T\subset J_D$ be the image of $Z$ by $j_D^{(2)}$.
This is a divisor with self intersection \[T.T=2.\]
The image $F(C)$ of $C$ by $F$ is algebraically equivalent to $l T$.
The divisors of poles of the functions 
 $\bs$,  $\bp$, $\bq$, and $\br$, are algebraically  equivalent to $2T$, $2T$, $6T$, and $4T$,  respectively.
Seen as functions on $C$,
the functions  $\bS(u)$,  $\bP(u)$, $\bQ(u)$, and $v\bR(u)$, thus have
degrees bounded by  $4l$, $4l$, $12l$, and $8l$, respectively.
So the rational fractions  $\bS$,  $\bP$, $\bQ$, and $\bR$,
have  degrees bounded by $2l$, $2l$, $6l$, and $4l+3$, respectively.
The four rational fractions $\bS$,
$\bP$, $\bQ$, $\bR$ provide
a compact description of the isogeny $f$ from
which we can deduce any desirable information about it.

\subsection{Associated differential system}\label{sec:ds}

The morphism $F : C\rightarrow J_D$ induces a map
\[F^* : H^0(J_D,\Omega^1_{J_D/\KK})\rightarrow H^0(C,\Omega^1_{C/\KK}).\]
So the vector $(\bS,\bP,\bQ,\bR)$
satisfies a first order differential system. This system
can be given a convenient form using local coordinates.
A basis for $H^0(C,\Omega^1_{C/\KK})$ is made
of $du/v$ and $udu/v$. We identify $H^0(J_D,\Omega^1_{J_D/\KK})$
with the invariant
subspace of $H^0(D\times D, \Omega^1_{D\times D/\KK})$  by the permutation
of the two factors. We deduce that a basis for this
space is made of  $dx_1/y_1+dx_2/y_2$ and $x_1dx_1/y_1+x_2dx_2/y_2$.
Let  $M = (m_{i,j})_{1\le i,j\le 2}$ be the matrix of $F^*$ with
respect to these two bases.  So
\begin{eqnarray}\label{eq:matf}
F^*(dx_1/y_1+dx_2/y_2)&=&(m_{1,1}+m_{2,1}\ttimes u)\ttimes du/v,\\
F^*(x_1dx_1/y_1+x_2dx_2/y_2)&=&(m_{1,2}+m_{2,2}\ttimes u)\ttimes du/v \nonumber.
\end{eqnarray}
Let $P = (u_P,v_P)$ be a point on $C$. We assume that
$v_P\not = 0$. Let $Q_1$ and $Q_2$ be two points on 
$D$ such that $F(P)$ is the class of $Q_1+Q_2-K_D$.
We assume that $F(P)\not =0$, so
the divisor $Q_1+Q_2$ is non-special. We also assume that
$Q_1\not = Q_2$ and either of the points are defined over $\KK$.
Let $t$ be a formal
parameter. Set $\LL= \KK((t))$. We  call \[P(t) = (u(t),v(t))\] the
point on $C(\LL)$ corresponding to the value $t$
of the local parameter $u-u_P$ at $P$. The image of $P(t)$
by $F$ is the class of $Q_1(t)+Q_2(t)-K_D$ where 
$Q_1(t)$ and $Q_2(t)$ are two $\LL$-points on $D$.
\begin{equation}\label{eq:comdiag}
\xymatrix{
\Spec \KK[[t]] \ar@{->}[rrr]^{t\mapsto (Q_1(t),Q_2(t))} \ar@{->}[d]^{t\mapsto  P(t)}&&&  D\times D \ar@{->}[d] \\
C \ar@{->}[rrr]^F&&& J_D.
}
\end{equation}

From Equations~(\ref{eq:matf}) and the commutativity of 
diagram~(\ref{eq:comdiag}) we deduce that the coordinates 
$(x_1(t),y_1(t))$ and $(x_2(t),y_2(t))$ of $Q_1(t)$ and $Q_2(t)$ satisfy the following
non-singular first order system of differential equations.
\begin{equation}\label{eq:edo}
\left\{
\begin{array}{ccc}\frac{\dotx_1(t)}{y_1(t)} + \frac{\dotx_2(t)}{y_2(t)}&=&
\frac{(m_{1,1}+m_{2,1}\ttimes u(t))\ttimes \, \dotu(t)}{v(t)},\\[\medskipamount]
\frac{x_1(t)\ttimes \, \dotx_1(t)}{y_1(t)} + \frac{x_2(t)\ttimes \, \dotx_2(t)}{y_2(t)}&=&
\frac{(m_{1,2}+m_{2,2}\ttimes u(t))\ttimes \, \dotu(t)}{v(t)},\\
y_1(t)^2 &=& h_D(x_1(t)),\\
y_2(t)^2 &=& h_D(x_2(t)).\\
\end{array}
\right.
\end{equation}
So we can recover the complete description of the isogeny,
namely the rational fractions $\bS$, $\bP$, $\bQ$, $\bR$, 
from the knowledge of the image by $F$ of  a {\it single} formal point
on $C$. More concretely, we compute
the image $\{Q_1(t),Q_2(t)\}$ of $P(t)$ by $G$ with low accuracy, then
deduce from Equation~(\ref{eq:edo}) the values of the four scalars 
$m_{1,1}$, $m_{1,2}$, $m_{2,1}$, $m_{2,2}$. Then use Equation~(\ref{eq:edo})
again to increase the accuracy of the formal expansions up to $O(t^{\OO l})$
and recover the rational fractions from their expansions using continued
fractions. 
Coefficients
of $x_1(t)$ and $x_2(t)$ can be computed  one  by one
using
Equation~(\ref{eq:edo}). Reaching
accuracy $\OO l$ then requires $\OO l^2$ operations in $\KK$.
We can also use more advanced methods \cite{bk,bostan} with quasi-linear
complexity in the expected accuracy of the result.
Both methods may produce    zero denominators
if the characteristic is small. In that case we use a trick introduced
by Joux and Lercier \cite{lj} in the context of elliptic curves. 
We lift to a $p$-adic field having $\KK$ as residue field. The denominators introduced by
(\ref{eq:edo}) do not exceed $p^{\OO \log(l)}$. The required $p$-adic accuracy, and the
impact on the
complexity are thus  negligible.

\subsection{Computing isogenies}\label{sec:ci}

We are given a curve  $C$ of genus two, a Weierstrass  point $O_C$
and a 
maximal isotropic subspace $V$ in $J_C[l]$.
We set  \[A=J_C/V.\]  Since $2O_C$ is a canonical
divisor  we set $\theta = O_C$. Using this Theta characteristic we define
  a principal  polarization 
 $Y$ on $A$ as 
in Section~\ref{sec:quo}. 
We
use the methods
given in Sections~\ref{sec:bl} and~\ref{sec:quo} to find nine
functions $\eta_0 =1$, $\eta_1$,~\ldots,~$\eta_8$, such 
that  $(\eta_0,\eta_1,\eta_2,\eta_3)$ is a basis of
$H^0(A,\cO_A(2Y))$ and
$(\eta_0,\ldots,\eta_8)$ is a basis of
$H^0(A,\cO_A(3Y))$.
We thus define two maps
$e_2 : A\rightarrow \PP^3$ and $e_3 : A\rightarrow \PP^8$.
Denoting $\pi : \PP^8\DashedArrow[->,densely dashed    ]{\PP^3}$ the projection
\[\pi(Z_0:Z_1:\dots:Z_8)=(Z_0:Z_1:Z_2:Z_3)\] we have
$\pi\circ e_3=e_2$. Evaluating the $(\eta_i)_{0\le i\le 8}$
at enough points we find equations for $e_3(A)$ and $e_2(A)$.
The intersection of $e_3(A)$ with the hyperplane 
$H_0$ with equation $Z_0=0$ in $\PP^8$ is $e_3(Y)$ 
counted with multiplicity $3$.
We now assume that $Y$ is a smooth and absolutely integral  curve of genus two. This is the generic case,
and it is true in particular  whenever the Jacobian
$J_C$ of $C$ is absolutely simple.
The intersection of $e_2(A)$ with the hyperplane  with equation
$Z_0=0$ in $\PP^3$ is $e_2(Y)$ counted with multiplicity $2$.
The map $Y\rightarrow e_2(Y)$ is the 
hyperelliptic quotient. It has degree two. 
Its image  $e_2(Y)$ is a plane curve of degree two. 
We deduce explicit equations for a hyperelliptic curve
$D$ and an isomorphism $D\rightarrow Y$.

We now define a rational map $\varphi$ from $J_C$ into
the symmetric square of $D\simeq Y$ by setting, for $z$
a generic point on $J_C$, 
\begin{equation}\label{eq:intersec}
 \varphi(z)=Y_{f(z)} \cap Y,
\end{equation}
where $Y_{f(z)}$ is the translate of $Y$ by $f(z)$.
Recall that  $O_C$ is a Weierstrass point on $C$.
We define a map $\psi$ from $C$ into
the symmetric square of $D\simeq Y$ by setting, for $P\in C$ 
a generic point, $\psi(P)=\varphi(P-O_C)$.
We check that $\psi(O_C)$ is a canonical divisor $K_Y$ on $Y$.
The difference $\psi(P) -\psi(O_C)$ is a degree
$0$ divisor on $Y$ and belongs to the class $f(P-O_C)$.
So  $\psi : C\rightarrow
Y^{(2)}$  is the map $G$ introduced  in Section~\ref{sec:alfo}.

We explain how to evaluate the map $\varphi$ at a given $z$
in $J_C$.
The main point is to compute the intersection in Equation~(\ref{eq:intersec}).
This is a matter of linear algebra. 
We pick two auxiliary classes $z_1$ and $z_2$ in $J_C$. We set $z'_1
=-z-z_1$ and $z'_2=-z-z_2$. We assume that $\varphi(z_1)$,
$\varphi(z_2)$, $\varphi(z'_1)$, $\varphi(z'_2)$ are pairwise disjoint.
Seen as a  function on $A=J_C/V$, the function 
$\eta_X[[z_1]+[z'_1]+[z]]$ belongs to $H^0(A,\cO_A(3Y))$. Evaluating it at a few points
we can express it as a linear combination of the elements $(\eta_i)_{0\le i\le 8}$
of our basis: 
\[\eta_X[[z_1]+[z'_1]+[z]] = \sum_{0\le i\le 8}c_i\ttimes \eta_i.\]
The hyperplane section $H_1$ with equation $\sum_ic_iZ_i = 0$ intersects $e_3(A)$ at 
$Y_{f(z_1)} \cup  Y_{f(z_1')}\cup Y_{f(z)}$. We similarly find an hyperplane 
section $H_2$ with equation $\sum_i d_iZ_i = 0$ intersecting  $e_3(A)$ at 
$Y_{f(z_2)} \cup Y_{f(z_2')} \cup Y_{f(z)}$. So 
\begin{equation*}
 \varphi(z)=Y_{f(z)} \cap Y = H_1\cap H_2 \cap H_0 \cap e_3(A),
\end{equation*}
is computed by linear substitutions.
Altogether we have proven the theorem below.
\begin{theorem}[Computing isogenies for genus two curves]\label{th:ci}
There exists a probabilistic (Las Vegas) algorithm that takes
as input 
\begin{itemize}
\item a finite field $\KK$ of odd
characteristic $p$, and cardinality $q$,
\item an odd prime $l$ different from $p$, 
\item a projective, smooth, absolutely irreducible curve of
genus two,  $C$, given by  a
plane  affine singular model
as  in Equation~(\ref{eq:g2}), 
\item  a maximal isotropic subgroup
$V$ in $J_C[l]$ as in Section~\ref{sec:trace}, 
such that the curve $Y$ introduced in Section~\ref{sec:quo} is smooth and absolutely integral.
\end{itemize}

The algorithm returns
a genus two curve $D$ and a map $F : C \rightarrow J_D$ as in 
Equation~(\ref{eq:defisog}). The running time is $l^{2+\egot(l)}\ttimes  
(\log q)^{1+\egot(q)}$.
\end{theorem}
In case $Y$ is not smooth and absolutely integral, it is a stable curve of genus two. The calculation above will work just as well and produce
one  map from $C$ onto  either of the components of $Y$. We do not formalize this degenerate case.


\section{An example}\label{sec:exa}

Let $\KK$ be the field with $1009$ elements. Let 
\[h_C(u) = u(u-1)(u-2)(u-3)(u-85)\in \KK[u]\] and let $C$ be the 
projective, smooth, absolutely irreducible  curve of genus
two given by the singular plane model
with equation $v^2 = h_C(u)$. Let $O_C$ be the place
at infinity. Let $o_C$ be the corresponding
class in $\Pic^1(C)$.  Let $T_1$ be the effective divisor of degree $2$
defined by the ideal \[(u^2+247u+67,v-599-261u)\subset \KK[u,v]/(v^2-h_C(u)).\]
Let $T_2$ be the effective divisor of degree $2$
defined by the ideal \[(u^2+903u+350,v-692-98u)\subset  \KK[u,v]/(v^2-h_C(u)).\]
The classes of $T_1-2O_C$ and $T_2-2O_C$ generate a totally isotropic
subspace $V$ of dimension $2$ inside $J_C[3]$. Let $A = J_C/V$.
Let $W_{-o_C}\subset J_C$ be the set of classes of divisors 
$P-O_C$ for $P$ a point on $C$. Since $O_C$ is a Weierstrass point, we 
have $[-1]^*W_{-o_C}=W_{-o_C}$. Let 
$X\subset J_C$ and $Y\subset A$ be the two divisors introduced 
at the beginning of Section~\ref{sec:quo}. 
Let $B\subset C$   be the effective divisor of degree $2$
defined by the ideal $(u^2+862u+49,v-294-602u)$. Let
$b\in J_C(\KK)$ be the class of $B-2O_C$. 
For
$i$ in $\{0,1, 2,3, 85\}$ let $P_i$ be the point
on $C$ with coordinates $u=i$ and $v=0$.
The class of $P_i$ in $\Pic^1(C)$ is  denoted $p_i$.
We set $p_\infty = o_C$ and $p_+=p_0+p_1-o_C\in \Pic^1(C)$.

For $i$ in  $\{\infty,0,1,+, 2,3, 85\}$ let $\eta_i$  be the unique 
function on $J_C$ with divisor
$2(X_{p_i-o_C}-X)$ and taking value $1$ at $b$.
These  functions are invariant by $V$ and may
be seen as level two Theta functions on $A$. 
Evaluating these  functions at a few points
we check that $(\eta_\infty,\eta_0,\eta_1,\eta_+)$  form a basis of 
$H^0(A,\cO_A(2Y))$ and 

\begin{eqnarray*}
\eta_2&=&437\eta_\infty+241\eta_0+332\eta_1,\\
\eta_3&=&294\eta_\infty+246\eta_0+470\eta_1,\\
\eta_{85}&=&639\eta_\infty+827\eta_0+553\eta_1.
\end{eqnarray*}

Call $Z_\infty$, $Z_0$, $Z_1$, $Z_+$ the projective
coordinates associated with $(\eta_\infty,\eta_0,\eta_1,\eta_+)$. The Kummer
surface of $A$ is defined  by the vanishing of the following   homogeneous form of degree four
\[\begin{split}
&597Z_\infty^2Z_0^2+14Z_\infty^2Z_0Z_1+781Z_\infty^2Z_0Z_++819Z_\infty^2Z_1Z_++835Z_\infty^2Z_1^2+615Z_\infty^2Z_+^2\\+&401Z_\infty Z_0^2Z_1+833Z_\infty Z_0^2Z_++553Z_\infty Z_0Z_1Z_++843Z_\infty Z_0Z_1^2+206Z_\infty Z_0Z_+^2+418Z_\infty Z_1^2Z_+\\
+&321Z_\infty Z_1Z_+^2
+796Z_0^2Z_1Z_++Z_0^2Z_1^2+1000Z_0^2Z_+^2+856Z_0Z_1^2Z_++655Z_0Z_1Z_+^2+555Z_1^2Z_+^2.\end{split}\]
This equation is found by evaluating all four functions at forty points.
We set $Z_{\infty}=0$ in this form  and find the square of the
following quadratic form
\begin{equation}\label{eq:coni}
611Z_0Z_{+} +  581{Z_1}{Z_{+}}-Z_0Z_1
\end{equation}
which is an equation for $e_2(Y)$ in the projective plane $Z_{\infty}=0$.
Recall $e_2 : A\rightarrow \PP^3$ is the map
introduced in Section~\ref{sec:ci}.
Set
\begin{eqnarray*}
Z_2&=&437Z_\infty+241Z_0+332Z_1\\
Z_3&=&294Z_\infty+246Z_0+470Z_1\\
Z_{85}&=&639Z_\infty+827Z_0+553Z_1.
\end{eqnarray*}
We find an affine parameterization of the conic $e_2(Y)$
in Equation~(\ref{eq:coni})
by setting \[Z_+ =1 \text{ and } Z_1 = x Z_0.\]
For $i$ in  $\{0,1,+, 2,3, 85\}$ call $D_i$ the line
with equations $\{Z_\infty = 0, Z_i = 0\}$.
There are six intersection points between $e_2(Y)$ and one
of the $D_i$. These are the six branched points of the
hyperelliptic  cover $Y\rightarrow e_2(Y)$. They
correspond to the values \[\{0,\infty,513,51,243,987\}\]
of the  $x$ parameter. We set 
\[h_D(x) = x(x-513)(x-51)(x-243)(x-987)\in \KK[x]\] and let $D$ be the 
genus two curve given by the singular plane model
with equation $y^2 = h_D(x)$. Let $O_D$ be the unique
place at infinity on $D$.
Let $P=(u,v)$ be a point on $C$. Using notation introduced in
Section~\ref{sec:alfo} we call $F(P)$ the image of $P-O_C$
in $J_D$ and $G(P)$ an  effective divisor  such that 
$F(P) = G(P)-2O_D$.
This divisor is defined by the ideal 
\[(x^2-\bS (u)x+\bP(u),y-v(\bT (u)+x\bR(u))\subset  \KK(u,v)[x,y]/(y^2-h_D(x))\]
where
\begin{eqnarray*}
\bS (u) &=&354 \frac{u^5+647u^4+931u^3+597u^2+73u+361}{u^5+832u^4+811u^3+215u^2+420u},\\
\bP(u)&=&50\frac{u^5+262u^4+812u^3+770u^2+868u+314}{u^5+832u^4+811u^3+215u^2
+420u},\\
\bR(u)&=&304\frac{u^6+437u^5+623u^4+64u^3+194u^2+3u+511}{u^8+239u^7+
983u^6+800u^5+214u^4+489u^3+191u^2},\\
\bT(u)&=&678\frac{u^6+697u^5+263u^4+895u^3+859u^2+204u+130}{u^8+239u^7+983u^6
+800u^5+214u^4+489u^3+191u^2}.
\end{eqnarray*}
We note that the fraction $\bQ(u)$ introduced in Section~\ref{sec:alfo}
is \[\bQ = h_C \ttimes (\bT^2+ \bR^2\ttimes \bP+\bS\ttimes \bR\ttimes \bT).\]
We now explain how these rational fractions were computed. We consider the formal point
\[P(t)= (u(t),v(t))=(832+t,361 + 10t + 14t^2+O(t^3))\] on $C$.
 We compute $G(P(t)) = \{Q_1(t),Q_2(t)\}$ and find 
\begin{eqnarray*}
Q_1(t)&=&(x_1(t),y_1(t))=(973 + 889t + 57t^2+O(t^3), 45 + 209t + 39t^2+O(t^3)),\\
Q_2(t)&=&(x_2(t),y_2(t))=(946 + 897t + 252t^2+O(t^3), 911 + 973t + 734t^2+O(t^3)).
\end{eqnarray*}
Using  Equation~(\ref{eq:edo}) we deduce the values \[m_{1,1}= 186, m_{1,2}=864, m_{2,1}=853, m_{2,2}=640.\]
Using Equation~(\ref{eq:edo}) again we increase the accuracy in the expansions for $x_1(t)$, $x_2(t)$, $y_1(t)$,
and $y_2(t)$ then deduce the rational fractions $\bS$, $\bP$, $\bR$, and $\bT$.

\bibliographystyle{plain}

\end{document}